\documentclass{article}
\usepackage[utf8]{inputenc}
\usepackage[margin=1in]{geometry}
\usepackage{commands}
\usepackage{color}
\usepackage{biblatex}
\usepackage{url}

\usepackage{breakurl}

\title{Lower bounds for superpatterns and universal sequences}
\author{Zachary Chroman\thanks{Massachusetts Institute of Technology, Cambridge, MA 02142. Email: \href{mailto:zacchro@mit.edu}{\nolinkurl{zacchro@mit.edu}}.}\and Matthew Kwan\thanks{
Department of Mathematics, Stanford University, Stanford, CA 94305.
Email:
\href{mattkwan@stanford.edu}{\nolinkurl{mattkwan@stanford.edu}}.
Research supported in part by SNSF project 178493 and NSF Award DMS-1953990.}
\and  Mihir Singhal\thanks{Massachusetts Institute of Technology, Cambridge, MA 02139. Email: \href{mailto:mihirs@mit.edu}{\nolinkurl{mihirs@mit.edu}}.}}
\date{}

\begin{document}

\maketitle

\begin{abstract}
A permutation $\sigma\in S_n$ is said to be \emph{$k$-universal} or a \emph{$k$-superpattern} if for every $\pi\in S_k$, there is a subsequence of $\sigma$ that is order-isomorphic to $\pi$. A simple counting argument shows that $\sigma$ can be a $k$-superpattern only if $n\ge (1/e^2+o(1))k^2$, and Arratia conjectured that this lower bound is best-possible. Disproving Arratia's conjecture, we improve the trivial bound by a small constant factor. We accomplish this by designing an efficient encoding scheme for the patterns that appear in $\sigma$. This approach is quite flexible and is applicable to other universality-type problems; for example, we also improve a bound by Engen and Vatter on a problem concerning $(k+1)$-ary sequences which contain all $k$-permutations.
\end{abstract}

\section{Introduction}

We say that a permutation $\sigma \in S_n$ \textit{contains} a permutation $\pi \in S_k$ \emph{as a pattern} when there exists some subset of indices $t_1<\dots<t_k$ such that $\sigma(t_i)<\sigma(t_j)$ if and only if $\pi(i)<\pi(j)$ (that is to say, $\sigma$ has a subsequence that is order-isomorphic to $\pi$). Starting with some foundational work by MacMahon~\cite{MacMahon} and later Knuth~\cite{Knuth}, the subject of permutation patterns has become quite vast over the years; see for example the books \cite{KitaevBook,BonaBook}, the surveys~\cite{StanleySurvey,SteingrimssonSurvey,VatterSurvey}, and Tenner's \emph{Database of Permutation Pattern Avoidance}~\cite{Tenner}.

If a permutation $\sigma$ contains all $k!$ patterns in $S_k$, it is said to be \textit{$k$-universal} or a \textit{$k$-superpattern}. Following some early work in the case $k=3$ by Simion and Schmidt~\cite{SimionSchmidt}, superpatterns were first introduced as a general notion by Arratia in 1999~\cite{Arratia}. Arratia raised perhaps the most fundamental question about superpatterns: how short can a $k$-superpattern be? He gave a very simple construction of a $k$-superpattern of length $k^2$, and observed that any $k$-superpattern $\sigma\in S_n$ must have length at least $n\ge (1/e^2+o(1))k^2$. This lower bound arises from the trivial inequality $\binom{n}{k}\ge k!$, which holds because a permutation $\sigma\in S_n$ has only $\binom{n}{k}$ different subsequences (and therefore at most this many patterns) of length $k$. Arratia conjectured that this trivial lower bound of $(1/e^2+o(1))k^2$ cannot be improved, meaning that there are $k$-superpatterns of length only $(1/e^2+o(1))k^2$.

Arratia's conjecture has remained open for the last twenty years, but various improvements have been made to the upper bound. Following earlier work by Eriksson, Eriksson, Linusson and Wastlund~\cite{Eriksson}, a clever construction due to Miller \cite{miller} achieves a $k$-superpattern of length $(k^2+k)/2$, and this was later improved slightly to $(k^2+1)/2$ by Engen and Vatter \cite{engenvatter}. Actually, the authors of \cite{Eriksson} made a conjecture (contradicting Arratia's conjecture) that $(1/2+o(1))k^2$ is the true minimum length of a $k$-superpattern.

Arratia's conjecture may at first seem a little unmotivated, but when placed in a wider context, it is quite natural. Generally speaking, a mathematical structure is said to be \emph{universal} if it contains all possible substructures, in some specified sense (this concept seems to have been first considered by Rado~\cite{Rado}). For many different notions of universality, there are trivial lower bounds arising from counting arguments in a very similar way to Arratia's superpattern bound, and these bounds are very often known or suspected to be essentially optimal. As a simple example, for any positive integers $k$ and $q$, there is known to exist a \emph{de Bruijn sequence} with these parameters, which is a string over a size-$q$ alphabet containing every possible length-$k$ string exactly once. A much less simple example is the case of universality for finite graphs, a subject that has received a large amount of attention over recent years due in part to applications in computer science (see for example \cite{AKTZ}). A graph is said to be \emph{$k$-induced-universal} if it contains an induced copy of all possible graphs on $k$ vertices. There are at least $2^{\binom k 2}/k!$ graphs on $k$ vertices, and an $n$-vertex graph has at most $\binom n k$ induced subgraphs with $k$ vertices, so a straightfoward computation shows that every $k$-induced-universal graph must have at least $(1+o(1))2^{(k-1)/2}$ vertices (this was first observed by Moon~\cite{Moon}). In a recent breakthrough, Alon \cite{AlonGraph} proved that this trivial lower bound is essentially optimal, by showing that a random graph with $(1+o(1))2^{(k-1)/2}$ vertices typically contains \emph{almost} all $k$-vertex graphs as induced subgraphs, and then proving that the remaining ``exceptional'' graphs can be handled separately without adding too many vertices.

It is very tempting to try to adapt Alon's proof strategy to permutations, to prove Arratia's conjecture. First, one might try to prove that a random permutation of length $n=(1/e^2+o(1))k^2$ typically contains almost all permutations of length $k$ (one imagines that the $\binom n k$ subsequences of a random permutation tend to yield mostly distinct patterns, so the trivial counting bound should not be too far from the truth). Second, one hopes to be able to deal with the few remaining permutations in some more ad-hoc way, without substantially increasing the length of the permutation. In this paper we prove that not only is Arratia's conjecture false, but even the first of these two steps fails.

\begin{thm}
\label{thm:original}
Suppose that $n < (1.000076/e^2)k^2$. Then every $\sigma \in S_n$ contains only $o(k!)$ different patterns $\pi\in S_k$. In particular, there is no $k$-superpattern of length less than $(1.000076/e^2)k^2$.
\end{thm}

We have made some effort to optimize the constant in \cref{thm:original}, where it would not negatively affect the readability of the proof. However, it would be quite complicated to fully squeeze the utmost out of our proof idea (see \cref{sec:concluding} for further discussion). In any case, obtaining a constant that is substantially larger than $1/e^2$ seems quite out of reach.

To very briefly describe the approach in our proof of \cref{thm:original}, it is instructive to reinterpret Arratia's trivial bound in an ``information-theoretic'' way. Namely, for each pattern $\pi \in S_k$ that appears in $\sigma$, note that $\sigma$ gives a way to encode $\pi$ by a set of indices $t_1 < \dots < t_k$. Since there are at most $\binom n k$ possible outcomes of this encoding, there are at most $\binom n k$ different patterns that we could have encoded, yielding Arratia's bound. The key observation towards improving this bound is that the aforementioned encoding is often slightly wasteful. For example, suppose that we are encoding a pattern $\pi$ that appears in $\sigma$ at indices $t_1<\dots<t_k$, and suppose that $t_{i+1}-t_{i-1}> k$ for some $i$. Now, having specified all the indices except $t_i$, there are more than $k$ possibilities for $t_i$. This means it would be ``cheaper'' simply to specify the relative value $\pi(i)$ itself, rather than to specify the index $t_i$ (note that there are only $k$ choices for this relative value, and $\pi(i)$ together with $t_1,\dots,t_{i-1},t_{i+1},\dots,t_k$ still fully determine $\pi$). More generally, we can design an encoding scheme that decides for each $i$ whether to specify the index $t_i$ or the relative value $\pi(i)$ depending on the size of $t_{i+1}-t_{i-1}$. Note that the average value of $t_{i+1}-t_{i-1}$ is about $2k/e^2$ (which is substantially less than $k$), so for most $i$ we will end up just specifying $t_i$, as in the trivial encoding scheme. Our small gain comes from the fact that for almost all choices of $t_1,\dots,t_k$, there are a small number of $i$ for which the difference $t_{i+1}-t_{i-1}$ is much larger than average (if the $t_1,\dots,t_k$ are chosen randomly, then the consecutive differences $t_{i+1}-t_{i}$ are well-approximated by independent exponential random variables). The details of the proof are in \cref{sec:original,sec:widthboundproof}.

At first glance the approach sketched above may seem to be an extremely general way to improve trivial counting-based lower bounds for basically any universality-type problem. However, for many interesting problems this approach does not seem to give us even tiny (lower-order) improvements. We discuss this in the concluding remarks (\cref{sec:concluding}).
We were, however, able to find another problem closely related to Arratia's conjecture, for which our approach is useful. To describe this, we generalize to the setting where $\sigma$ is a sequence of (not necessarily distinct) elements of $[r]:=\{1,\dots,r\}$. We define containment of a pattern identically: if $\sigma \in [r]^n$ is a sequence and $\pi \in S_k$ is a pattern, then $\sigma$ contains $\pi$ if there exist indices $t_1 < \dots < t_k$ such that $\sigma(t_i)<\sigma(t_j)$ if and only if $\pi(i)<\pi(j)$. Again, $\sigma$ is said to be \textit{$k$-universal} if it contains all $k$-patterns. Note that any $k$-universal sequence yields a $k$-superpattern of the same length, since it is possible to find a permutation with the same relative ordering (breaking ties arbitrarily).

Engen and Vatter, in their survey paper \cite{engenvatter}, considered the question of finding the shortest possible $k$-universal sequence in the cases $r = k, k+1$. In the case $r=k$, it is known that the answer must be $(1+o(1))k^2$, as proved by Kleitman and Kwiatkowski~\cite{KleitmanKwiatkowski}. The case $r = k+1$ is of particular interest since Miller's construction of a $k$-superpattern of length $n=(k^2+k)/2$ actually comes from a $k$-universal sequence in $[k+1]^n$. Meanwhile, until now the best lower bound was the trivial one obtained as follows. Writing $a_m$ for the number of occurrences of each $m\in [r]$ in $\sigma$, the number of subsequences of $\sigma$ with elements $m_1, m_2, \cdots, m_k$ is equal to $a_{m_1}a_{m_2} \dots a_{m_k}$. This expression is at most $(n/k)^k$ by convexity. Thus, the total number of subsequences that don't contain repeated elements is at most $\binom{r}{k} (n/k)^k$. This must be at least $k!$ for $\sigma$ to be a $k$-universal sequence, so as long as $r = (1+o(1))k$ we must have $n \ge (1/e+o(1))k^2$.
We improve upon this trivial bound as follows.

\begin{thm}
\label{thm:k+1}
Suppose that $n<(1+e^{-600}) k^2/e$. Let $\sigma \in [(1+e^{-1000})k]^n$ be a sequence of length $n$. Then, $\sigma$ can only contain $o(k!)$ patterns in $S_k$.
\end{thm}

The proof of \cref{thm:k+1} appears in \cref{sec:k+1}. We made no effort to optimize the constants.

\section{Lower-bounding the length of a superpattern} \label{sec:original}

In this section we will prove \cref{thm:original}, modulo a simple probabilistic lemma that will be proved in the next section. We assume, as we may, that $k$ is odd, and we will generally omit floor symbols in quantities that may not be integers.

Let $\sigma\in S_n$ be a permutation of length $n < 1.000076 k^2/e^2$. Our objective is to prove that $\sigma$ contains $o(k!)$ patterns. As outlined in the introduction, the idea is to define an efficient encoding; namely, an injective function $\phi$ from the set of all patterns in $\sigma$ into a set of size $o(k!)$. Our encoding is always going to specify the positions (in $\sigma$) of the odd-indexed entries of $\pi$, but for for the even indexed-entries, we are going to encode either the position or the relative value, depending on which of the two ``requires less information'' (unless this encoding scheme does not provide significant savings, in which case we just encode the positions of all entries of $\pi$).

To describe our encoding more formally, let $\pi$ be a pattern of length $k$ that is contained in $\sigma$. 
Fix a set of indices $T = \{t_1, \dots, t_k\}$, where $t_1 < \dots < t_k$, such that the values of $\sigma$ on $t_1, \dots, t_k$ form the pattern $\pi$. For even $1 < i < k$, define $b_i = t_{i+1} - t_{i-1}$ to be the \textit{width} around $t_i$ (so there are $(k-1)/2$ widths). Note that the widths only depend on the indices $t_i$ for odd $i$. Also, note that if one knows $t_{i-1}$ and $t_{i+1}$, then the number of possibilities for $t_i$ is $b_i-1$. That is to say, if the width around $t_i$ is large, then specifying $t_i$ is very ``expensive.'' Now, to define our encoding $\phi$ we split into cases. With foresight, define $c = 0.00075$ and $d = 8.180$.

{\bf Case 1:} If fewer than $ck$ of the widths are at least $dn/k$, then there is not much to be gained by a creative encoding scheme, and we simply encode $\pi$ by the indices $t_1,\dots,t_k$ (formally, set $\phi(\pi)=(t_1,\dots,t_k)$).

{\bf Case 2:} Otherwise, at least $ck$ of the widths are at least $dn/k$. Let $I$ be a set of exactly $ck$ even integers such that $b_i\ge dn/k$ for each $i\in I$ (the specific choice of $I$ can be basically arbitrary, but we make sure it only depends on the indices $t_i$ where $i$ is odd). Then, encode $\pi$ by specifying $\pi(i)$ for each $i\in I$, and specifying $t_i$ for each $i\notin I$. Formally, we can define $\phi(\pi)$ to be the triple $\left(I,(t_i)_{i\notin I},(\pi(i))_{i\in I}\right)$.

It is important to note that our encoding $\phi$ is injective: in both cases, it is possible to recover $\pi$ by knowing $\sigma$ and $\phi(\pi)$. To see this in Case 2, note that if we know $I$ and $(\pi(i))_{i\in I}$, then we know the set of values of $\pi(i)$ for $i\notin I$ (but not their relative order). This relative order is obtainable from the relative order of the $\sigma(t_i)$, for $i\notin I$.

Now, in order to prove \cref{thm:original}, we need to show that there are not too many possibilities (substantially fewer than $\binom n k$) for the value of $\phi(\pi)$. The main technical ingredient is the fact that Case 1 is needed only rarely, as follows. Let $\lambda=0.999924$.

\begin{lem} \label{lem:widthbound}
Let $\sigma\in S_n$ be a permutation of length $n < 1.000076 k^2/e^2$, with $n=\Omega(k^2)$, and assume that $k$ is sufficiently large. For all but $O(\lambda^k) \binom{n}{k}$ of the size-$k$ subsets $T\subseteq[n]$, at least $ck$ of the widths are at least $dn/k$.
\end{lem}

Before proving \cref{lem:widthbound}, we first show how to use it to deduce \cref{thm:original} via the encoding $\phi$. We return to \cref{lem:widthbound} in \cref{sec:widthboundproof}.

\begin{proof}[Proof of \cref{thm:original}]
First, note that we may assume that $n=\Omega(k^2)$ (since we can arbitrarily extend $\sigma$ to a longer permutation without destroying any of the patterns it contains). Now, using the assumption $n < (1.000076/e^2)k^2$, the fact that $1.000076< 1/0.999924$, the standard inequality $\binom n k\le (en/k)^k$, and Stirling's approximation, we can see that $\lambda^k \binom{n}{k} \le \lambda^k(en/k)^k\le (k/e)^k = o(k!)$. So, \cref{lem:widthbound} shows there are only $o(k!)$ possible outcomes of $\phi(\pi)$ arising from Case 1.

Next, we consider Case 2. Let $\Phi$ be the set of all outcomes of $\phi(\pi)$ obtainable from Case 2, so our objective is to show that $|\Phi|=o(k!)$. We first want to show that there are not many possibilities for the pair $(I,(t_i)_{i\notin I})$ in the definition of $\phi$. Intuitively, this should be the case because we specifically decided to ``forget'' the indices $t_i$ for which the width $t_{i+1}-t_{i-1}$ is large, and these indices are the most expensive to remember.

Note that our definition of the widths of a pattern $\pi$ only depended on the subset $T$ of indices where $\pi$ appeared. So, let $\mathcal F$ be the set of all size-$k$ subsets $T\subseteq[n]$ with at least $ck$ widths of size at least $dn/k$. For every $T=\{t_1,\dots,t_k\}\in \mathcal F$, we can choose a set $I$ of $ck$ even integers such that the widths $b_i$, for $i\in I$, are at least $dn/k$ (here $I$ only depends on the odd-numbered indices of $T$, just as in the definition of $\phi$). Let $\Psi$ be the set of all possibilities for the pair $(I,(t_i)_{i\notin I})$, among all $T\in \mathcal F$.

Now, note that we can encode the sets in $\mathcal F$ using the encoding $\psi(T)=(I,(t_i)_{i\notin I},(t_i)_{i\in I})$. Given any $(I,(t_i)_{i\notin I})\in \Psi$, the number of ways to extend this to a valid encoding $\psi(T)$ is at least
\[
\left(\frac{dn}{k} -  1\right)^{ck},\]
because for each $i\in I$ there are at least $dn/k-1$ possibilities for $t_i$, given $t_{i-1}$ and $t_{i+1}$ (here we are using the restriction that $I$ only consists of even-numbered indices). We deduce that $|\Psi|(dn/k -  1)^{ck}\le |\mathcal F|\le \binom{n}k$, so $|\Psi|\le \binom n k (dn/k -  1)^{-ck}$.

Given any pair $(I,(t_i)_{i\notin I}) \in \Psi$, the number of ways to choose $(\pi(i))_{i\in I}$ to extend our pair to an encoding $\phi(\pi)$ of some $\pi$ is at most ${k!}/{(k-ck)!}$. So,
\[
|\Phi|\le |\Psi| \frac{k!}{(k-ck)!}.
\]
Combining this with our upper bound for $|\Psi|$, we compute that $|\Phi|$ is bounded above by
\begin{align*}
    \frac{k!}{(k-ck)!} \left(\frac{dn}{k} - 1\right)^{-ck} \binom{n}{k}
    &\leq O(1)\,k!\left(\frac{1}{(k(1-c)/e)^{(1-c)k}} \left(\frac{dn}{k}\right)^{-ck} 
    \left(\frac{n e}{k}\right)^k\right)\\
    &= O(1)\,k!\left(\left(\frac{n}{(k^2/e^2)}\right)^{1-c}e^{c}  d^{-c}(1-c)^{-(1-c)} \right )^k=o(k!),
\end{align*}
as desired (recall that $n < 1.000076k^2/e^2$, $c = 0.00075$ and $d = 8.180$, and note that $(dn/k-1)^{-ck}/(dn/k)^{-ck}=(1-O(1/k))^{-ck}=O(1)$).  All in all, we have proved that our encoding $\phi$ can output only $o(k!)$ different values, so the number of distinct patterns in $\sigma$ must be $o(k!)$.

\end{proof}

\section{Large gaps in random sets} \label{sec:widthboundproof}
In this section we prove \cref{lem:widthbound}, showing that for all but $O(\lambda^k)\binom n k$ choices of $T$, at least $ck$ of the widths are at least $dn/k$ (where $d=8.180$, $c=0.00075$ and $\lambda=0.999924$). The details are a bit technical, but the intuition is fairly simple, as follows. If we choose $T$ randomly, then its elements approximately correspond to a Poisson point process with intensity $k/n$ in the interval $[0,n]$. So, the lengths of the gaps between the elements of $T$ are approximately exponentially distributed with mean $n/k$. Each of the widths is a sum of a pair of these gap lengths, so its distribution is approximately $\operatorname{Gamma}(2,k/n)$. The distribution function of $\operatorname{Gamma}(2,k/n)$ is $x\mapsto 1-e^{-x k/n}(x k/n+1)$ (for $x\ge 0$), so we expect about an $e^{-d}(d+1) > 2c$ fraction of the $(k-1)/2$ widths to have size at least $dn/k$.

\begin{proof}[Proof of \cref{lem:widthbound}]
Let $T$ be a uniformly random size-$k$ subset of $[n]$, and denote its elements by $t_1<\dots<t_k$ (so each $t_i$ is a random variable). We show that the probability that fewer than $ck$ widths are at least $dn/k$ is at most $O(\lambda^k)$.

Define the \textit{gaps} $a_0, a_1, \dots, a_k$ by $a_i = t_{i+1} - t_i$ (where $a_0 = t_1$ and $a_k = n+1-t_k$). Then the widths can be represented as $b_i=a_i+a_{i-1}$. Note that $(a_0, \dots, a_k)$ is uniformly distributed over all sequences of $k+1$ positive integers summing to $n+1$.

Let $(X_0, \dots, X_k)\in \mathbb R^{k+1}$ be a point uniformly distributed in the simplex given by
$$\Gamma = \{(x_0, \dots, x_k)\;:\; x_0, \dots, x_k \ge 0,\;x_0 + \dots + x_k = n + 1\}.$$
Let $a_i' = \lfloor X_i \rfloor$ be obtained by rounding down $X_i$, for each $0 \le i \le k$. Then $(a_1',\dots,a_k')$ is stochastically dominated by $(a_1,\dots,a_k)$, in the sense that the two sequences can be coupled such that we always have $a_i\ge a_i'$ for each $i$. Let $b_i'=a_i'+a_{i-1}'$, and let $B_i=X_i+X_{i-1}$. It suffices to show that with probability $1-O(\lambda^k)$, we have $B_i\ge dn/k+2$ (therefore $b_i'\ge dn/k$) for at least $ck$ different $i$.

Now, it is known that the distribution of $(X_0, \dots, X_k)$ is identical to the distribution of
$$\frac{n+1}{\xi_0 + \dots + \xi_k}(\xi_0, \dots, \xi_k),$$
where $\xi_i$ are i.i.d.\ exponential random variables with rate 1 (see for example \cite[Section~4.1]{Dirichlet}).
Let $d' = 8.282 > d$.
By a Chernoff bound for sums of exponential random variables (see for example \cite{chernoff}), we have 
\[
\Pr\left[\xi_0 + \dots + \xi_k \ge \frac{d'}{d}(k+1)\right] 
\le 
\exp\left({-(k+1)\left(\frac{d}{d'} - 1 -\ln\left(\frac{d}{d'}\right)\right)}\right) \le O(\lambda^k).
\]

Also, each of the $\xi_{i-1} + \xi_i$ (for even $i$) are i.i.d.\ with $\Pr(\xi_{i-1} + \xi_i \ge x) = e^{-x}(1+x)$. Take $d'' = 8.283$ slightly larger than $d'$, and let $p=e^{-d''}(1+d'')>2c$, so that
for each $i$ we have $\Pr(\xi_{i-1} + \xi_i\ge d'')\ge p$. Thus, by a Chernoff bound for the binomial distribution (see for example \cite[Theorem~A.1.13]{AlonSpencer}), the probability that $\xi_{i-1} + \xi_i \ge d''$ for fewer than $ck$ different $i$ is at most
\[\Pr[\Bin((k-1)/2, p) \le ck] \le \exp\left(\frac{-(p(k-1)/2-ck)^2}{2\cdot p(k-1)/2}\right)\le O(\lambda^k). 
\]
Combining the above two bounds, we conclude that with probability $1-O(\lambda^k)$, we have \[B_i=\frac{\xi_i+\xi_{i-1}}{\xi_1+\dots+\xi_k}\ge \frac{d''}{d'}\cdot \frac{d(n+1)}{k+1}\ge \frac{dn}k + 2,\]
for at least $ck$ different $i$ (using the assumptions that $k$ is large and $n=\Omega(k^2)$), as desired.
\end{proof}

\section{Lower-bounding the length of a universal sequence} \label{sec:k+1}

In this section we prove \cref{thm:k+1}, modulo a probabilistic lemma that will be proved in the next section. It actually suffices to prove the following seemingly weaker result, where we do not allow the alphabet size to be greater than $k$.
\begin{prop}
Suppose $\sigma \in [k]^n$, where $n < (1 + e^{-600})k^2/e$.
Then, for sufficiently large $k$, the number of $k$-patterns in $\sigma$ is at most $\exp(-e^{-600}k)\,k!$.
\label{prop:k}
\end{prop}

Before proving \cref{prop:k}, we deduce \cref{thm:k+1} from it.

\begin{proof}[Proof of \cref{thm:k+1}]

Suppose $k$ is sufficiently large, and let $t=e^{-1000}k$. Consider $\sigma \in [k+t]^n$. We call the elements of $[k+t]$ \emph{symbols}. Every pattern $\pi \in S_k$ that appears in $\sigma$ uses some set $Y$ of $k$ symbols, and therefore it appears in the subsequence $\sigma_Y\in Y^{n'}$ of $\sigma$ obtained by keeping only the symbols in $Y$.

Now, by \cref{prop:k}, for each of the $\binom{k+t}{k}$ choices of $Y$, there are at most $\exp(-e^{-600}k) k!$ patterns in $\sigma_Y$. Thus, the total number of patterns contained in $\sigma$ is at most
\[
    \binom{k+t}{k}\exp(-e^{-600}k)\,k!
    \leq \exp(-e^{-600}k) \frac{(k+t)^{k+t}}{k^k t^t}\,k!.
\]

Recalling that $t / k = e^{-1000}$, one can verify that this expression is $o(k!)$, as desired.
\end{proof}

Now we proceed to the proof of \cref{prop:k}. Let $\sigma\in [k]^n$ be a sequence of length $n<(1 + e^{-600})k^2/e$, and let $a_1, \dots, a_k$ be the numbers of occurrences of the symbols $1, \dots, k$ in $\sigma$. Recall from the introduction that we have a trivial upper bound of $a_1a_2 \dots a_k$ on the number of patterns $\pi\in S_k$ that appear in $\sigma$; we would like to improve on this by designing an efficient encoding scheme $\phi$ for the patterns that appear in $\sigma$.

Let $\pi \in S_k$ be a pattern that is contained in $\sigma$, and consider indices $t_1, \dots, t_k$ that represent $\pi$ in $\sigma$, such that $\sigma(t_i)=i$ for each $i$. (This is different from the proof of \cref{thm:original} where $(t_1,\dots,t_k)$ was an increasing sequence of indices, but note that $(t_1, \dots, t_k)$ still uniquely determines $\pi$.) Let $c = e^{-290}$. Now, we encode $\pi$ as follows: first, we specify $t_1, \dots, t_{k-ck}$. Then, instead of specifying $t_m$ for $k - ck < m \le k$, we will instead specify the relative position of $t_m$, with respect to all the other $t_i$. Formally, we let $\psi(m)$ be the binary vector $(\boldsymbol{1}_{t_m<t_i})_{i< m}$ indicating the position of $t_m$ relative to all $t_i$ with $i<m$, and define $\phi(\pi)$ to be the pair $\left((t_i)_{i\le k-ck},(\psi(m))_{m>k-ck}\right)$. This encoding is injective.

To see why this encoding should be more efficient than the trivial one, it is helpful to consider the extreme case where we first specify $t_1, \dots, t_{k-1}$. Typically, the occurrences of the last symbol $k$ are not going to be perfectly distributed between the $t_i$, and there are multiple possible outcomes of $t_k$ that have the same position $\psi(k)$ relative to the other $t_i$. That is to say, specifying the relative position $\psi(k)$ of $t_k$ should be cheaper than specifying $t_k$ exactly.

To make precise the above intuition, we need to make some definitions. Fix $k - ck < m \le k$, and suppose we have specified indices $t_1,\dots,t_{k-ck}$ that divide the interval $[n]$ into $k-ck+1$ disjoint subintervals. Suppose that in $\sigma$, the symbol $m$ appears in positions $s_1< \dots< s_{a_m}$. Then the number of possible positions to place $t_m$ relative to $t_1, \dots, t_{k-ck}$ equals the number of subintervals that contain some $s_j$. Equivalently, this number of possible positions is one more than the number of adjacent pairs $s_j, s_{j+1}$ that are ``split'' by some $t_i$ between them. In this case, we say that the pair $s_j, s_{j+1}$ constitutes an \textit{$m$-split with respect to $t_i$}. Our main technical ingredient is an upper bound on the number of $m$-splits for $k-ck < m \le k$. We also give ourselves the freedom to permute the symbols in $\sigma$, which allows us to choose the most convenient $ck$ symbols as the ``last ones.''

\begin{lem} \label{lem:expsplits}
Consider $\sigma \in [k]^n$ with $n < (1 + e^{-600})k^2/e$. Either the conclusion of \cref{prop:k} holds for trivial reasons, or else it is possible to permute the symbols of $\sigma$ such that the following two conditions hold.
\begin{enumerate}[(1)]
    \item For all $k - ck < m \le k$, we have $a_m \ge 0.1k$.
    \item For all but $k^{O(1)} \exp(-e^{-560}k)\,a_1 \dots a_{k - ck}$ choices of $t_1, \dots, t_{k-ck}$ (where each $\sigma(t_i)=i$), the following key property holds: For each $m>k-ck$, the total number of $m$-splits with respect to $t_1, \dots, t_{k-ck}$ is at most $a_m (1-e^{-280})$.
\end{enumerate}

\end{lem}

Since the proof of \cref{lem:expsplits} is rather technical, we defer it to the end of this section after deducing \cref{prop:k}.

\begin{proof}[Proof of \cref{prop:k}]
We will assume that the symbols of $\sigma$ have been permuted to satisfy \cref{lem:expsplits}. Our main goal is to show that there are at most $k^{O(1)} \exp(-e^{-580}k) a_1 \dots a_k$ possible outcomes of the encoding $\phi(\pi)$; we will then be able to deduce the desired bound by convexity.

Consider a pattern $\pi\in S_k$ appearing in $\sigma$ at indices $t_1,\dots,t_k$. First, we deal with the case where $t_1,\dots,t_{k-ck}$ do not satisfy the key property in \cref{lem:expsplits}. There are at most $\exp(-e^{-560}k) a_1 \dots a_{k-ck}$ such possibilities for $t_1,\dots,t_{k-ck}$. Then, we just trivially observe that there are at most $a_m$ ways to choose $t_m$, for all $m>k-ck$. So, the total number of possibilities for $\phi(\pi)$ among such $\pi$ is at most $k^{O(1)} \exp(-e^{-560}k) a_1 \dots a_{k}$, which is substantially less than our target.

Next, we consider the case where $t_1,\dots,t_{k-ck}$ do satisfy our key property. This implies that for each $m>k-ck$, there are at most $a_m (1-e^{-280})+1$ distinct possibilities for the position of $t_m$ relative to $t_1,\dots,t_{k-ck}$. If we additionally want to specify the position of $t_m$ relative to $t_{k-ck+1},\dots,t_{m-1}$, there are at most $m-(k-ck+1)$ additional ways to make this choice. So, using condition (1) of \cref{lem:expsplits}, and recalling that $c = e^{-290}$, the total number of possibilities for the relative position $\psi(m)$ of $t_m$ is at most
\begin{align*}
a_m(1 - e^{-280}) + m - (k-ck) &\le a_m(1 - e^{-280}) + ck \\
&\le a_m(1 - e^{-280} + 10c) \\
&< a_m(1 - e^{-290}).
\end{align*}
Thus, since there are at most $a_i$ ways to pick $t_i$ for each $i \le k-ck$, the number of possibilities for $\phi(\pi)$ among $\pi$ satisfying our key property is at most
\begin{align*}
\left( \prod_{i=1}^{k-ck} a_i \right) \left( \prod_{m>k-ck} a_m(1-e^{-290}) \right) 
&= (1-e^{-290})^{ck} a_1 \dots a_k \\
&\le \exp(-e^{-290}ck) a_1 \dots a_k \\
&\le \exp(-e^{-580}k) a_1 \dots a_k.
\end{align*}
All in all, accounting for both cases (when the key property is satisfied, and when it is not), the total number of possibilities for $\phi(\pi)$ is at most
\[\exp(-e^{-580}k) a_1 \dots a_k + k^{O(1)} \exp(-e^{-560}k) a_1 \dots a_k \le k^{O(1)} \exp(-e^{-580}k) a_1 \dots a_k.\]
We have $a_1 + \dots + a_k = n$, so by convexity, the number of possibilities for $\pi$ is at most

\[
k^{O(1)} \exp(-e^{-580}k) \left(\frac nk \right)^k 
= k^{O(1)} \left(\exp(-e^{-580}) \frac{n}{k^2/e}\right)^k k!\le \exp(-e^{-600}k) k!,
\]

for sufficiently large $k$, as desired. This completes the proof of \cref{prop:k}.
\end{proof}

\section{Few splits for random indices}
In this section we prove \cref{lem:expsplits}. First, we want to be able to assume that most symbols occur fairly often. Call a symbol $m$ \textit{common} if $a_m > 0.1k$.

\begin{lem}\label{lem:good}
Consider $\sigma \in [k]^n$ with $n < (1 + e^{-600})k^2/e$. Either the conclusion of \cref{prop:k} holds for trivial reasons, or else at least $0.99k$ symbols are common.
\end{lem}
\begin{proof}
Suppose without loss of generality that $1,\dots,0.01k$ are all not common. Then the number of $k$-patterns in $\sigma$ is at most \[a_1\cdots a_k \leq \left(0.1k\right)^{0.01k} a_{0.01k+1}\cdots a_k\leq (0.1k)^{0.01k}\left(\frac{n}{0.99k}\right)^{0.99k},\]
by convexity, since $a_{0.01k+1}+\dots+a_k \leq n$. 
We can rewrite this latter expression as
\begin{align*}
    \left(\frac{0.1^{0.01}}{0.99^{0.99}}\right)^k k^k \left(\frac{n}{k^2}\right)^{0.99k} \leq k^k \left(\frac{0.988n}{k^2}\right)^k \leq  \left(\frac{0.988n}{k^2/e}\right)^k k!\le \exp(-e^{-600}k)\,k!,
\end{align*} 
so the conclusion of \cref{prop:k} holds.
\end{proof}

Next we will need a structural lemma. Suppose that in $\sigma$, the symbol $m$ appears in positions $s_1< \dots< s_{a_m}$. An \emph{$m$-gap} is the (possibly empty) interval between a pair of adjacent indices $s_j, s_{j+1}$. We say that an $m$-gap is \textit{full} if there exists some symbol $m' \neq m$ for which the gap contains at least $0.9a_{m'}$ occurrences of $m'$. In this case we say that the $m$-gap is \textit{filled} by $m'$. See \cref{fig:fullgapsex} for an example illustrating full gaps.

\begin{lem} \label{lem:fullgaps}
Consider $\sigma \in [k]^n$ with $n < (1 + e^{-600})k^2/e$, and suppose that at least $0.99k$ symbols are common. Then there are at least $0.03k$ common $m$ for which the number of full $m$-gaps is less than $0.9a_m$. 
\end{lem}

\begin{figure}[h]
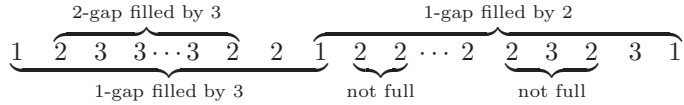

    \centering
    \[\lefteqn{\underbrace{\phantom{1\quad1\quad1\quad1\quad1\quad1\quad11\quad11}}_{\text{1-gap filled by 3}}}1\quad\overbrace{2\quad3\quad 3\cdots3\quad2}^{\text{2-gap filled by 3}}\quad2\quad\overbrace{1\;\; \underbrace{2\quad 2}_{\text{not full}}\cdots \; 2\quad \underbrace{2\quad3\quad2}_{\text{not full}}\quad 3\quad1}^{\text{1-gap filled by 2}}\] 
    \caption{An example showing gaps that are full and not full}
    \label{fig:fullgapsex}
\end{figure}

\begin{proof}
Consider some common $m$ for which the number of full $m$-gaps is at least $0.9a_m > 0.09k$ (if there is no such $m$ then we are immediately done). Note that each $m'$ can only fill at most one $m$-gap, so at least $0.9a_m - 0.01k > 0.08k$ of these full $m$-gaps are filled by different common symbols. Let $S$ be a set of $0.08k$ common symbols $m'$ that each fill a different $m$-gap.

Now, for any $m'\in S$, inside the $m$-gap filled by $m'$, there are at least $0.9a_{m'}$ instances of $m'$. These instances themselves form at least $0.9a_{m'} - 1$ different $m'$-gaps; denote the set of such $m'$-gaps by $G_{m'}$. Now, note that all the different gaps in $\bigcup_{m'\in S} G_{m'}$ are pairwise disjoint, since all $m'$-gaps are disjoint from each other for any fixed $m'$, and further, each $G_{m'}$ is entirely contained in a different $m$-gap. Thus at most $k$ of the gaps in $\bigcup_{m'\in S} G_{m'}$ can be full (each of the $k$ symbols can fill at most one of these gaps). It follows that for at least $|S|-k/20=0.03k$ of the symbols $m'\in S$, at most 20 of the gaps in $G_{m'}$ are full (otherwise, there would be more than $20(k/20)=k$ such full gaps). For each of these symbols $m'$, the total number of full $m'$-gaps is at most $0.1a_{m'} + 20$, because there are at most 20 full $m'$-gaps in $G_{m'}$, and $G_{m'}$ contains at least $0.9a_{m'} - 1$ of the $a_{m'}-1$ different $m'$-gaps that exist. Since each such $m'$ is common, we have $a_{m'} > 0.1k$, so $0.1a_{m'} + 20 < 0.9a_{m'}$ for large $k$. The desired conclusion follows, since these $m'$ now satisfy the condition of \cref{lem:fullgaps}
\end{proof}

Now we are ready to prove \cref{lem:expsplits}.

\begin{proof}[Proof of \cref{lem:expsplits}]
Consider $\sigma \in [k]^n$ with $n < (1 + e^{-600})k^2/e$. By \cref{lem:good} and \cref{lem:fullgaps}, after a permutation of the symbols we may assume that the last $0.03k$ symbols $m>0.97k$ are all common, and each have the property that the number of full $m$-gaps is less than $0.9a_m$. It is already clear that condition (1) of the lemma statement is satisfied (recall that $c = e^{-290} < 0.03$). 

Now, for each $i\le k-ck$, let $t_i$ be a uniformly random index satisfying $\sigma(t_i)=i$ (independently for each $i$). Fix some $m>k-ck$, and let $X$ be the random variable counting the number of $m$-splits with respect to $t_1, \dots, t_{k-ck}$. It suffices to show that with probability at least $1-k^{O(1)}\exp(-e^{-560}k)$, we have $X\le a_m(1-e^{-280})$. We will then be able to take the union bound over all $m>ck$.

As before, suppose that in $\sigma$, the symbol $m$ appears in positions $s_1, \dots, s_{a_m}$ (in increasing order). Let $J$ be the set of all $j$ such that the $m$-gap between $s_j$ and $s_{j+1}$ is not full. Then, since we are assuming that the number of full $m$-gaps is less than $0.9a_m$, we have $|J|\ge 0.1a_m$. Also, for each $i \le k - ck$ and $j \in J$, let $b_{ij}$ be the number of occurrences of $i$ in the $m$-gap between $s_j$ and $s_{j+1}$. Note that since the gap between $s_j$ and $s_{j+1}$ is not full, we always have $b_{ij}/a_i \le 0.9$.

For any $j\in J$, there is an $m$-split at $s_j, s_{j+1}$ precisely when the $m$-gap between $s_j, s_{j+1}$ contains one of $t_1,\dots,t_{k-ck}$. The probability that this event does not occur is
\[\prod_{i=1}^{k-ck}\left(1-\frac{b_{ij}}{a_i}\right) \ge  \exp\left(-2.6\sum_{i=1}^{k-ck}\frac{b_{ij}}{a_i}\right),\]
where the inequality follows from the fact that each $b_{ij}/a_i \le 0.9$. We deduce that
\[\E X\le a_{m} - 1 - \sum_{j \in J} \exp\left(-2.6\sum_{i=1}^{k-ck}\frac{b_{ij}}{a_i}\right).\]
Recalling that $\sum_j b_{ij} \le a_i$ for each $i$, and using convexity, we further deduce that
\begin{align*}
\E X\le a_{m} - 1 - |J| \exp\left(-2.6 \frac{1}{|J|} \sum_{j \in J} \sum_{i=1}^{k-ck}\frac{b_{ij}}{a_i}\right) 
&\le a_{m} - |J| \exp\left(-2.6 \frac{1}{|J|} \sum_{i=1}^{k-ck} \sum_{j \in J}\frac{b_{ij}}{a_i}\right) \\
&\le a_{m} - |J| \exp\left(\frac{-2.6(k-ck)}{|J|}\right).
\end{align*}
Now, recall that $|J|\ge 0.1a_m\ge 0.01k$ (since $m$ is common), so we can compute that $\E X< a_m\left(1 - e^{-270}\right).$
Finally, note that if any $t_i$ (for $1 \le i \le k - ck$) is changed, then this can cause the number of $m$-splits $X$ to increase or decrease by at most one. By the Azuma--Hoeffding inequality (see for example~\cite[Theorem~7.2.1]{AlonSpencer}), we conclude that
\begin{align*}
\Pr(X\ge a_m(1-e^{-280}))&\le\exp\left(-\frac{(a_m(e^{-270}-e^{-280}))^2}{2(k-ck)}\right)
\le \exp(-e^{-560}k),
\end{align*}
as desired (we have used that $a_m\ge 0.1k$, since $m$ is common).
\end{proof}

\section{Concluding remarks} \label{sec:concluding}

In this paper we have introduced a new method to prove lower bounds for universality-type problems, by identifying local inefficiencies in trivial encoding schemes. We have used this method to improve bounds on two different problems. It would be very interesting to find further applications of this idea. For example, Alon showed that the minimum number of vertices in a $k$-induced-universal graph is asymptotic to the trivial lower bound, but he observed that the trivial lower bound is not exactly tight, and raised the question of better understanding lower-order terms (see \cite[Section~5]{AlonGraph}). There are also many other problems about universality in graphs (for example, universality with respect to containment of trees or bounded-degree graphs; see for example \cite{AKTZ}), where it is not yet known whether trivial lower bounds are asymptotically tight.

However, there seem to be some difficulties in applying our methods to graph problems. Roughly speaking, the reason that we were able to obtain improvements in the setting of permutation patterns is as follows. For a pattern $\pi\in S_k$, the amount of information (entropy) carried by a single value $\pi(t)$ is about $\log (k/e)=\log k-1$ ``on average'' (because the information carried by $\pi$ itself is $\log(k!)\approx k\log(k/e)$). Taking a different point of view, there are $k$ possibilities for $\pi(k)$, so when viewed in isolation, the amount of information carried by $\pi(k)$ is $\log k$. The first point of view is relevant for computing trivial lower bounds, and the second point of view is relevant for improving local inefficiencies. It was important for our proof strategy that these two points of view gave very similar answers. However, in most graph problems, these two points of view tend to give quite different answers: specifying the adjacencies of a single vertex requires much more information than the ``average information per vertex,'' and therefore our methods do not seem to be directly applicable to graphs. It would be interesting to investigate this further.

On the subject of superpatterns, obviously there is still a large gap between our new bound and the upper bound $(1/2+o(1)) k^2$ obtained by Miller. It should be clear from the proof of \cref{thm:original} that it is possible to make various small improvements to our lower bound: for example, it was convenient to restrict our attention to widths $t_{i+1}-t_{i-1}$ only for even $i$, but with a more sophisticated argument one could take both even and odd $i$ into account. Also, the bounds in \cref{lem:widthbound} were rather crude, and presumably one could prove exact large deviation bounds for the number of widths above a given threshold. However, it would be very complicated to fully optimize all aspects of our argument, and it seems unlikely that one could prove a lower bound much larger than $k^2/e^2$ without substantial new ideas. At present, we do not have a strong conjecture for the true minimal length of a $k$-superpattern.

It is worth mentioning some related problems that may be more tractable, and may shed light on the true minimal length of a $k$-superpattern. For example, instead of demanding that our permutation $\sigma$ contains every pattern of length $k$, we can ask for permutations that contain \emph{almost} every pattern of length $k$. Stronger upper bounds are known in this case: a construction by Eriksson, Eriksson, Linusson and Wastlund~\cite{Eriksson} captures all but an exponentially small proportion of $k$-patterns in a permutation of length $(1+o(1))k^2/4$. We remark that this construction represents an obstruction to counting-based arguments: even if Miller's upper bound of $(1+o(1))k^2/2$ turns out to be best-possible (as conjectured in \cite{Eriksson}), a proof of this would have to be sensitive to the difference between containing \emph{all} patterns, and containing \emph{almost all} of them. As suggested by He and Kwan~\cite{hekwan}, it would also be very interesting to explore for which $n$ a \emph{random} permutation of length $n$ contains almost all $k$-patterns. It could be that this holds for $n$ quite close to $k^2/e^2$, and by analogy to Alon's study of universal graphs, this would suggest that there is also a $k$-superpattern of approximately the same length.

The study of pattern containment in random permutations is also of independent interest. It is a celebrated result in probability theory that the longest increasing subsequence in a random permutation of length $n$ is typically about $2\sqrt n$ (see for example~\cite{RomikBook}). This tells us that the ``threshold'' value of $n$, above which a random $n$-permutation is likely to contain the increasing pattern $1\,2\dots k\in S_k$, is approximately $k^2/4$. It would be interesting to understand how the threshold for containment of a pattern $\pi\in S_k$ depends on the structure of $\pi$. Alon (see \cite{Arratia}) conjectured that this threshold is never more than $(1/4+o(1))k^2$; in fact he conjectured that $(1/4+o(1))k^2$ is the threshold for being a $k$-superpattern. The best known bounds are still quite far from this conjecture: recently He and Kwan~\cite{hekwan} proved that a random permutation of length $n = 2000 k^2 \log \log k$ is typically a $k$-superpattern.

In a somewhat different direction, for each $\pi\in S_k$ one can also ask about the length of the longest $\pi$-free subsequence in a random permutation of length $n$. See \cite[Conjecture~1]{albert} for an interesting conjecture along these lines.

Finally, most of the above considerations are also relevant for pattern containment in $r$-ary sequences. Though Kleitman and Kwiatkowski~\cite{KleitmanKwiatkowski} found the asymptotics of the minimal length of a $k$-universal $k$-ary sequence, it is not obvious how the situation changes if $r>k$ or if one only requires containment of almost all patterns. There may also be interesting related problems about random sequences.

{\bf Acknowledgments:} we would like to thank Noga Alon and Xiaoyu He for insightful discussions. We would also like to thank the referees for their careful reading of the manuscript and their valuable comments.

\printbibliography[heading=bibintoc]

\end{document}